\documentclass{amsart}
\usepackage{amsmath, amsthm, amssymb, amsfonts}
\usepackage[utf8]{inputenc} 
\usepackage[T1]{fontenc}    
\usepackage{nicefrac}       
\usepackage{microtype}      
\usepackage{lipsum}
\usepackage{mathtools}
\usepackage{enumitem}
\usepackage{float}
\usepackage{graphicx}
\graphicspath{ {./images/} }

\newcommand{\xfrac}[2]{#1/#2}
\newcommand{\bfrac}[2]{(#1)/#2}
\renewcommand{\Im}{\operatorname{Im}}

\newtheorem{theorem}{Theorem}
\newtheorem{lemma}{Lemma}
\newtheorem{corollary}{Corollary}

\theoremstyle{definition}

\title{On C. Michel's hypothesis about the modulus of typically real polynomials}

\begin{document}

\author{Dmitriy Dmitrishin}
\address{Odessa National Polytechnic University, 1 Shevchenko Ave.,
Odessa 65044, Ukraine}
\email{dmitrishin@opu.ua}

\author{Andrey Smorodin}
\address{Odessa National Polytechnic University, 1 Shevchenko Ave.,
Odessa 65044, Ukraine}
\email{andrey.v.smorodin@opu.ua}

\author{Alex Stokolos}
\address{Georgia Southern University, Statesboro GA, 30460}
\email{astokolos@geogriasouthern.edu}

\begin{abstract}
Extremal problems for typically real polynomials go back to a paper
by W.~W.~Rogosinski and G.~Szeg\H{o}, where a number of problems were posed, which were partially solved 
by using orthogonal polynomials. 
Since then, not too many new results on extremal properties of typically real polynomials have been obtained.
Fundamental work in this direction is due to
M.~Brandt, who found a novel way of solving extremal problems.
In particular, he solved 
 C. Michel's problem of estimating the modulus of a typically real polynomial of odd degree.  
On the other hand, D.~K.~Dimitrov showed the effectivity of Fej\'er's method
for solving the Rogosinski--Szeg\H{o} problems.
In this article, we completely solve Michel's problem by using Fej\'er's method.
\end{abstract}

\keywords{Typically real polynomials, extremal trigonometric polynomials, Fej\'er method}

\maketitle

\section{Introduction}
 Let $T_{N}$ denote the set of all typically real in the unit disk $\mathbb{D}=\{z:|z|<1\}$ polynomials $F_{N}(z)=z+\sum^N_{j=2}c_{j}z^{j}$  of degree
 $N$ normalized by $F_{N}(0)=0$, $F'_{N}(0)=1$. The task is to find
\begin{equation}\label{e1}
  J_{N}=\max_{F_{N}(z) \in T_{N}} \sup_{z\in\mathbb{D}}\,|F_{N}(z)|=\max_{F_{N}(z) \in T_{N}} \max_{t}\,|F_{N}(e^{it})|.
\end{equation}
 Evidently $|F_{N}(e^{it})| \leq 1+\sum^{N}_{j=2}|c_{j}|$, so
 $$
  J_{N}\leq\max_{\alpha_{1}=1} \bigg\{ \sum_{j=1}^{N}|\alpha_{j}|:\sum_{j=1}^{N}\alpha_{j}\sin jt\geq0,\,t \in [0, \pi] \bigg\}.
 $$
If all coefficients of the extremal sine polynomial are not negative then the last
inequality becomes an equality. Here we use the fact that $F_N(z)=z+\sum^N_{j=2}\alpha_{j}z^{j}$ based on coefficients of non-negative on $[0,\pi]$ sine polynomial is typically real.

 In \cite{1} it was conjectured that
\begin{equation}\label{e2}
  J_N \leq \frac14\csc^2\frac{\pi}{2(N+2)}.
\end{equation}
This was proved in \cite{2}, and it was shown that for $N$ odd, \eqref{e2} turns into an equality.
Moreover, an extremal polynomial was exhibited:
\begin{equation}\label{e3}
  \frac{z}{z^2+1-2z\cos{\frac{\pi}{N+2}}}-\frac{4\cos^2\frac{\pi}{2(N+2)}}{N+2}(1+z^{N+2})\frac{1-z}{1+z}  \\
  \bigg(\frac{z}{z^2+1-2z\cos\frac{\pi}{N+2}}\bigg)^2,
\end{equation}
but its uniqueness was not considered.

In this paper we prove that the extremal polynomial is unique (up to signs of even degree terms) for all $N$,
 and for even $N$ we have
 $$
 J_N=\frac{1}{4\nu^2},
 $$
 where $\nu$ is the least positive root of $U'_{N+1}(x)$, where
  $U_j(x)=U_j(\cos t)=\frac{\sin(j+1)t}{\sin t}=2^j x^j + \cdots$ is a Chebyshev polynomial of the second kind. 
We also provide an algorithm for finding the coefficients of the extremal polynomial, and an explicit formula for the coefficients for $N$ odd.

Note that for $N$ odd, $\sin \frac{\pi}{2(N+2)}$ is the least positive root of
$U_{N+1}(x)$.
 
\section{Preliminaries}
 
Let
\begin{equation}\label{eq20}
A =
 \begin{pmatrix}
   0 & \frac{1}{2}  & 0 & 0 &  \cdots  \\
   \frac{1}{2}  & 0  &  \frac{1}{2} & 0 & \cdots \\
   0 &  \frac{1}{2}  & 0 & \frac{1}{2} & \cdots \\
   0  & 0  & \frac{1}{2} & 0 & \cdots \\
  \cdots & \cdots & \cdots & \cdots & \cdots 
 \end{pmatrix},\quad
 B =
 \begin{pmatrix}
  0  &  0 & \frac{1}{2}  & 0 & \cdots  \\
  0 & 0  & 0 & \frac{1}{2} & \cdots  \\
  \frac{1}{2} & 0 & 0 &0 & \cdots \\
  0 & \frac{1}{2} & 0 &0 & \cdots \\
  \cdots & \cdots & \cdots & \cdots  \\
 \end{pmatrix}
\end{equation}

be symmetric $N\times N$ matrices, and  $I$ the $N\times N$ unit matrix.

\begin{lemma}\label{lem2.1}
We have
$$
\det(4x^2(I+A)-(I-B))=\frac{1}{2^{N+2}x}U_{N+1}(x)U'_{N+1}(x).
$$
\end{lemma}
\begin{proof}
The matrices $\{c_{ij}\}_{i,j=1}^{N}$ and $\left\{(-1)^{i+j}c_{ij}\right\}_{i,j=1}^{N}$ have equal determinants, because each term $\pm c_{i_1 j_1}c_{i_2 j_2} \times \cdots \times c_{i_N j_N}$ in the decomposition of the determinant contains an even number of factors for which the sum of the row and column numbers is odd.
Indeed,
$(-1)^{i_1+j_1}$ $(-1)^{i_2+j_2}\cdots(-1)^{i_N+j_N}=(-1)^{2(1+2+\cdots+N)}=1$. 
Hence $\det\left(4x^2(I+A)-(I-B)\right)=\det\left(4x^2(I-A)-(I-B)\right)$. 

But \cite{3} gives a formula for  the latter determinant as in the statement.
\end{proof}

In the lemma below and throughout, we follow the terminology of Gantmacher \cite[Chapter X, \S 6]{11}.

\begin{lemma}\label{lem2.2}
The characteristic numbers of the matrix pencil $I+A-\lambda(I-B)$ are 
$$
{\left\{ \frac{1}{4\mu_j^2} \right\}}_{j=1}^{\lfloor\bfrac{N+1}{2}\rfloor},\quad
{\left\{ \frac{1}{4\nu_j^2} \right\}}_{j=1}^{N-\lfloor\bfrac{N+1}{2}\rfloor},\ \text{where}\ \mu_j=\cos\frac{j\pi}{N+2}, \ U'_{N+1}(\nu_j)=0;
$$
they can be arranged as
$$
\frac{1}{4\mu_1^2} < \frac{1}{4\nu_1^2} < \cdots < \frac{1}{4\nu^2_{\bfrac{N-1}{2}}} < \frac{1}{4\mu^2_{\bfrac{N+1}{2}}}
$$
if $N$ is odd, and
$$
\frac{1}{4\mu_1^2} < \frac{1}{4\nu_1^2} < \cdots < \frac{1}{4\mu^2_{\xfrac{N}{2}}} < \frac{1}{4\nu^2_{\xfrac{N}{2}}}
$$
if $N$ is even.
\end{lemma}
\begin{proof}
The determinant $\det\left(4x^2(I+A)-(I-B)\right)$ is a polynomial of degree $2N$ with zeros
${\left\{\pm\mu_j\right\}}_{j=1}^{\lfloor\bfrac{N+1}{2}\rfloor}$, ${\left\{\pm\nu_j\right\}}_{j=1}^{N-\lfloor\bfrac{N+1}{2}\rfloor}$,
where
 $\mu_j=\cos\frac{j\pi}{N+2}$, $U'_{N+1}(\nu_j)=0$, which can be arranged so that
$$
-\mu_1<-\nu_1<\cdots<-\nu_{\bfrac{N-1}{2}}<-\mu_{\bfrac{N+1}{2}}\\<0<\mu_{\bfrac{N+1}{2}}<
\nu_{\bfrac{N-1}{2}}<\cdots<\nu_1<\mu_1
$$
if $N$ is odd, and
$$
-\mu_1<-\nu_1<\cdots<-\mu_{N/2}<-\nu_{N/2}<0<\nu_{N/2}<
\mu_{N/2}<\cdots<\nu_1<\mu_1
$$
if $N$ is even, since the zeros of $U_{N+1}(x)$ and $U'_{N+1}(x)$ are interlaced.The characteristic numbers $\{\lambda_j\}^N_{j=1}$ of
$I+A-\lambda(I-B)$ are related to the zeros of $\det\left(4x^2(I+A)-(I-B)\right)$ by
$\lambda_{2j-1}=\frac{1}{4\mu_j^2}, \lambda_{2j}=\frac{1}{4\nu_j^2}$, $j=1,\ldots,\lfloor\bfrac{N+1}{2}\rfloor$, which gives the assertion.
\end{proof}

We now find the eigenvectors of the pencil $I+A-\lambda(I-B)$.

\begin{lemma}\label{lem2.3}
The principal eigenvectors $Z_j^{(0)}$ of the matrix pencil
$I+A-\lambda(I-\nobreak B)$, corresponding to the eigenvalues $\lambda_j=\frac{1}{4}\sec^2\frac{\pi j}{N+2}$,  are given by 
$Z_j^{(0)}=Z^{(0)}\bigl(\cos\frac{\pi j}{N+2}\big)$, $j=1,\ldots,\lfloor\bfrac{N+1}{2}\rfloor$, where
$$
Z^{(0)}(x)=\left\{(-1)^kU_{k-1}(x)U_k(x)\right\}_{k=1}^N.
$$
\end{lemma}
\begin{proof}
Consider the vector $\hat{Z}^{(0)}(x)=\left\{U_{k-1}(x)U_k(x)\right\}_{k=1}^N$. In \cite{3} it was proved that
$\left(4x^2(I-A)-(I-B)\right)\hat{Z}^{(0)}(x)=0$ for
$x=\cos\frac{j\pi}{N+2}$, $j=1,\ldots,\lfloor\bfrac{N+1}{2}\rfloor$.
Consequently,
\begin{multline*}
\left(\left(I+A\right)-\lambda_j\left(I-B\right)\right)Z^{(0)}(x)  \\
= -\frac{1}{4\cos^2\frac{\pi j}{N+2}}\left(4\cos^2\frac{\pi j}{N+2}(I-A)-(I-B)\right) \hat{Z}^{(0)}(x)=0.\qedhere
\end{multline*}
\end{proof}

\begin{lemma}\label{lem2.4}
Let $\nu_j$ be a positive root of the polynomial $U'_{N+1}(x)$. 
The principal vectors $Z_j^{(1)}$ corresponding to the characteristic numbers
$\lambda_j=\frac{1}{4\nu_j^2}$ of the matrix pencil $I+A-\lambda(I-B)$ are given by $Z_j^{(1)}=Z^{(1)}(\nu_j)$, $ j=1,\ldots,\lfloor\xfrac{N}{2}\rfloor$, where
\begin{multline}\label{e4}
Z^{(1)}(x)=\biggl\{(-1)^{k-1}( U_{k-1}(x)U_k(x)-U_{N-k}(x)U_{N-k+1}(x) \\
                          +\frac{N+1-2k}{N+1}U_N(x)U_{N+1}(x) )\bigg\}_{k=1}^N.
\end{multline}
\end{lemma}
\begin{proof}
Denote
\begin{multline*}
z_k(x)=(-1)^{k-1}\biggl( U_{k-1}(x)U_k(x)-U_{N-k}(x)U_{N-k+1}(x)  \\
            +\frac{N+1-2k}{N+1}U_N(x)U_{N+1}(x) \bigg).
\end{multline*}
Writing the system $\left(4x^2(I+A)-(I-B)\right)Z^{(1)}(x)=0$ coordinatewise we can see directly that all the left-hand sides of the equations of the system, from the second up to the penultimate equation, are identically zero for all~$x$.  The first equation (coinciding with the last one) is
\begin{equation}\label{e5}
(4x^2-1)z_1(x)+2x^2z_2(x)+\frac{1}{2}z_3(x)=0.
\end{equation}
We use the formula (see \cite{3})
\begin{equation}\label{e6}
\frac{d}{dx}U_k(x)=\frac{1}{2(1-x^2)}\left((k+2)U_{k-1}(x)-kU_{k+1}(x)\right).
\end{equation}
After some calculations we obtain
\begin{align*}
(4x^2-1)z_1(x)+2x^2z_2(x)+\frac{1}{2}z_3(x)&=-U_{N+1}(x)\left(\frac{N+2}{N+1}U_N(x)-xU_{N+1}(x)\right) \\
                                                                      &=\frac{-1+x^2}{N+1}U_{N+1}(x)U'_{N+1}(x).
\end{align*}
Consequently, the right-hand side of \eqref{e5} is zero at $x=\nu_j$.
\end{proof}

\begin{corollary}\label{cor2.4.1}
The principal vectors corresponding to the characteristic numbers $\lambda_j=\frac{1}{4}\sec^2\frac{\pi j}{N+2}$
of the matrix pencil $I+A-\lambda(I-B)$ can be written in the form $Z_j^{(0)}= (1/2) Z^{(1)}\bigl(\cos\frac{\pi j}{N+2}\big)$, $j=1,\ldots,
\lfloor\bfrac{N+1}{2}\rfloor$, where the vector $Z^{(1)}(x)$ is given by \eqref{e4}.
\end{corollary}

\begin{lemma}\label{lem2.5}
For all $k=1,\ldots,N$,
$$
2\sum_{j=1}^{N-k}U_j(x)U_{j+k-1}(x)=\frac{1}{1-x^2}\left((N-k)T_{k-1}(x)-T_{N+2}(x)U_{N-k-1}(x)\right),
$$
where $T_j(x)=T_j(\cos t)=\cos jt=2^{j-1}x^j+\cdots$ is a Chebyshev polynomial of the first kind.
\end{lemma}
\begin{proof}
Set $\xi=e^{it}$. Put $x=\cos t$, i.e. $x=\frac{1}{2}(\xi+\xi^{-1})$, 
$U_k(x)=\frac{\xi^{k+1}-\xi^{-(k+1)}}{\xi-\xi^{-1}}$, $T_k(x)=\frac{1}{2}(\xi^k+\xi^{-k})$ and use the formula for the sum of a geometric progression.
\end{proof}

\begin{corollary}\label{cor2.5.1}
Let $\mu=\cos\frac{\pi(N+1)}{2(N+2)}$. For all $k=1,\ldots,N$,
$$
2\sum_{j=1}^{N-k}U_j(\mu)U_{j+k-1}(\mu)=\frac{1}{1-\mu^2}\bigl((N-k)\left(\mu U_{k-2}(\mu)-U_{k-3}(\mu)\right)+U_{k+1}(\mu)\big).
$$
\end{corollary}

\begin{proof} This follows from  the relations $T_{n+1}=xU_n(x)-U_{n-1}(x)$ and
 $U_{N-j}(\mu)=(-1)^{\bfrac{N-1}{2}}U_j(\mu)$.
\end{proof}

We now localize the least positive root of  $U'_{N+1}(x)$.

\begin{lemma}\label{lem2.6}
Suppose $N$ is even and  $\nu$ is the minimal positive root of $U'_{N+1}(x)$. Then $\nu\in(\cos \tau_1, \cos \tau_2)$,
where $\tau_1=\frac{(N+1)\pi}{2(N+2)}$ and $ \tau_2=\frac{N\pi}{2(N+1)}$.
\end{lemma}
\begin{proof}
The polynomial $U_{N+1}(x)$ does not change its convexity on the interval $\bigl(0, \cos\frac{N\pi}{2(N+2)}\big)$, in particular on
 $(\cos\tau_1, \cos\tau_2)$. Consequently, $U'_{N+1}(x)$ is monotone on that interval. We use formula \eqref{e6}.
 The signs of $U'_{N+1}(x)$ at the endpoints of the latter interval coincide with the signs of the function $\Phi(t)=(N+3)\sin\,(N+1)t-(N+1)\sin\,(N+3)t$
at the endpoints of $(\tau_2, \tau_1)$.
We compute $\Phi(\tau_1)=(-1)^{\xfrac{N}{2}}2\sin\frac{\pi}{2(N+2)}$ and $\Phi(\tau_2)=-(-1)^{\xfrac{N}{2}}(N+1)\sin\frac{\pi}{N+1}$,
so the signs of $U'_{N+1}(\cos\tau_1)$ and $ U'_{N+1}(\cos\tau_2)$ are different, completing the proof.
\end{proof}

Let us describe some properties of the coordinates of the principal vector~\eqref{e4}.

\begin{lemma}\label{lem2.7}
For  $N$ even, let $\nu$ be the least positive root of $U'_{N+1}(x)$, and set
\begin{multline*}
z_k(x)=(-1)^{k-1}\biggl( U_{k-1}(x)U_k(x)-U_{N-k}(x)U_{N-k+1}(x) \\ 
          +\frac{N+1-2k}{N+1}U_N(x)U_{N+1}(x) \bigg),\quad k=1,\ldots,N.
\end{multline*}
Then
\begin{enumerate}[label=(\alph*)]
  \item $z_k(\nu)=z_{N-k+1}(\nu)$, $k=1,\ldots,\xfrac{N}{2}$;
  \item $z_{k+1}(\nu)>z_k(\nu)>0$, $k=1,\ldots,\xfrac{N}{2}-1$.
\end{enumerate}
\end{lemma}
\begin{proof}
Assertion (a) is evident. Consider now the differences $\Delta_k=z_{k+1}(\nu)-z_k(\nu)$, $k=1,\ldots,\xfrac{N}{2}-1$. We compute
\begin{equation*}
  \begin{aligned}
    \Delta_k=(-1)^k\bigg( & U_{k-1}(\nu)U_{k}(\nu) + U_{k}(\nu)U_{k+1}(\nu) - U_{N-k+1}(\nu)U_{N-k}(\nu) -  \\ 
                                   -\, &U_{N-k}(\nu)U_{N-k-1}(\nu) + \frac{N+1-2k}{N+1}U_N(\nu)U_{N+1}(\nu) + \\
                                  +\, & \frac{N-1-2k}{N+1}U_{N}(\nu)U_{N+1}(\nu)
                            \bigg).
  \end{aligned}
\end{equation*}
Grouping terms and applying the formulas $U_{k-1}(x)+U_{k+1}(x)=2xU_{k}(x)$ and $U^2_{k}(x)-U^2_{N-k}(x)=-U_{N+1}(x)U_{N-2k-1}(x)$ we obtain
$$
\Delta_{k}=(-1)^{k}2U_{N+1}(\nu)\left(-\nu U_{N-2k-1}(\nu)+\frac{N-2k}{N+1}U_{N}(\nu)\right).$$ 
The identity
$(N+3)U_{N}(\nu)-(N+1)U_{N+2}(\nu)=0$ yields $\frac{N+2}{N+1}U_{N}(\nu)=\nu U_{N+1}(\nu)$, so finally
$$
\Delta_{k} = (-1)^{k} 2 \nu U_{N+1}(\nu) \left( \frac{N-2k}{N+2} U_{N+1}(\nu) - U_{N-2k-1}(\nu) \right).
$$
Consider the auxiliary functions
$$\varphi_{k}(t) = (-1)^{k}\psi_{k}(t)\sin\,(N+2)t,
$$
where $$\psi_{k}(t) = \frac{1}{N+2}\sin\,(N+2)t -\frac{1}{N-2k}\sin\,(N-2k)t, \quad k = 1,\ldots,\xfrac{N}{2}-1.
$$ 
Set $\tau=\cos\nu$. The signs of $\varphi_{k}(\tau)$
and $\Delta_{k}$ coincide. Let us find the former. 
Since $(-1)^{\xfrac{N}{2}}U_{N+1}(\nu)>0$, we have $(-1)^{\xfrac{N}{2}}\sin\,(N+2)\tau>0$. 
We compute
$\psi'_{k}(t) = -2\sin\,(N+1-k)t\cdot\sin\,(k+1)t$. 
Note that $\cos\frac{N\pi}{2(N+1)}$ is the least positive  root of 
$U_{N}(x)$. Since $N$ is even,  all the second kind Chebyshev polynomials $U_k(x)$, $k=1,\ldots,N$, have constant sign on the interval $\bigl(0, \cos\frac{N\pi}{2(N+1)}\big)$,
  and $(-1)^{j-1}U_{2j-1(x)}>0$, $(-1)^{j}U_{2j}(x)>0$, $j=1,\ldots,\xfrac{N}{2}$. This means that the function  $\psi_k(t)$ is monotone on
 $\bigl(\frac{N\pi}{2(N+1)},\frac{\pi}{2}\big)$. As $\psi_k\left(\xfrac{\pi}{2}\right)=0$, each of these functions has constant sign on  that interval, and the signs of $\psi_k(t)$ alternate as $k$ varies. If $\xfrac{N}{2}$ is even, this sequence of signs starts from a minus, and for $\xfrac{N}{2}$ odd, from a plus. Consequently, $\varphi_k(t)>0$ for
$t \in \bigl( \frac{N\pi}{2(N+1)}, \frac{\pi}{2} \big)$, which yields $\varphi_k(\tau)>0$ and $\Delta_k>0$ for $k=1,\ldots,\xfrac{N}{2}-1$. If we additionally set $z_0=0$, we finally get $0=z_0<z_1<\cdots<z_{\xfrac{N}{2}}$.
\end{proof}

\section{An auxiliary problem}

Together with \eqref{e1} consider an auxiliary problem: find
\begin{equation}\label{e7}
\hat{J}_N = \max_{a_1=1}\Big\{ \sum_{j=1}^{N} \alpha_j : \sum_{j=1}^{N} \alpha_j\sin jt \geq 0,\, t \in [0, \pi] \Big\}.
\end{equation}

\begin{theorem}
We have
$$\hat{J}_N=
\begin{dcases}
\frac{1}{4}\csc^2\frac{\pi}{2(N+2)} &\text{if $N$ is odd,}\\
\frac{1}{4}\csc^2\vartheta &\text{if $N$ is even,}
\end{dcases}
$$
where $\vartheta$ is the least positive root of the equation
$$
(N+3)\cos\,(N+1)\vartheta+(N+1)\cos\,(N+3)\vartheta=0.
$$
\end{theorem}
\begin{proof}
Let $F_{N}^{(0)}(z)=\sum_{j=1}^{N}\alpha_jz^j$ be an extremal polynomial for 
\eqref{e7}. Then  $\hat{J}_N=F_{N}^{(0)}(1)$.
Further
$\Im( F_{N}^{(0)}( e^{it} ) ) = \sin t + \sum_{j=2}^{N}\alpha_j\sin jt = \beta_0\sin t\cdot( 1+ 2 \sum_{j=1}^{N-1} \beta_{j}\cos jt)$, where

\begin{equation}\label{e8}
\begin{aligned}[c]
&\beta_0 = 1 + \sum_{j=1}^{\lfloor \bfrac{N-1}{2} \rfloor} \alpha_{2j+1},\quad \beta_1=\frac{1}{\beta_0}\sum_{j=1}^{\lfloor \xfrac{N}{2} \rfloor} \alpha_{2j}, \\ 
&\beta_2 = \frac{1}{\beta_0}\sum_{j=1}^{\lfloor \bfrac{N-1}{2} \rfloor}\alpha_{2j+1},\ \ldots ,\ \beta_{N-1}=\frac{1}{\beta_0}\alpha_N.
\end{aligned}
\end{equation}
The trigonometric polynomial $1+2\sum_{j=1}^{N-1}\beta_j \cos jt$ is nonnegative on $[0, \pi]$, so
\begin{align*}
J_N = F_N^{(0)}(1) = & \sup_{\alpha_j}\Big\{ 1+\alpha_2+\alpha_3+\cdots: 1+2 \sum_{j=1}^{N-1} \beta_j \cos jt  \geq 0 \Big\}  \\ 
                               =  & \sup_{\beta_j}\bigg\{ \frac{1+\beta_1}{1-\beta_2}: 1+2 \sum_{j=1}^{N-1}\beta_j \cos jt \geq 0 \bigg\}.
\end{align*}
By the Fej\'er--Riesz theorem \cite{4} every nonnegative trigonometric polynomial can be represented as the square of the modulus of a trigonometric polynomial,  $| d_1+ d_2e^{it}+ \cdots+ d_Ne^{i(N-1)t} |^2 = 1 + 2 \sum_{j=1}^{N-1} \beta_j \cos jt$. Consequently,
\begin{equation}\label{e9}
\begin{split}
&d_1^2 + \cdots + d_N^2 = 1,\quad d_1d_2 + \cdots + d_{N-1}d_N = \beta_1,\\ 
&d_1d_3 + \cdots + d_{N-2}d_N  
= \beta_2,\ \ldots,\ d_1d_N = \beta_{N-1}.
\end{split}
\end{equation}
Then we can write $$\hat J_N = \max_{d_j}\left\{ \frac{1+d^T Ad}{1-d^T Bd}: d^T d = 1 \right\},$$ where $d = (d_1, \ldots, d_N)^T$ ($T$ denotes transposition),
A \eqref{eq20} is the $N\times N$ matrix of the quadratic form
$d_1^2 + \cdots + d_N^2 = 1$, and B \eqref{eq20} is the $N\times N$ matrix of the quadratic form $d_1d_2 + \cdots+ d_{N-1}d_N = \beta_1$.
Consequently,
$$\hat J_N = \max_{d_j}\left\{ \frac{1+d^T Ad}{1-d^T Bd} \right\} = \max_{d_j}\left\{ \frac{d^T(I+A)d}{d^T(I-B)d} \right\},$$
where $I$ is the unit matrix. The matrices $I+A$ and $I-B$ are clearly positive definite.
The problem $\hat J_N = \max_{d_j}\bigl\{ \frac{d^T(I+A)d}{d^T(I-B)d} \big\}$ can be reduced to finding generalized eigenvalues \cite{5}. Let
$\lambda_1\leq \cdots \leq \lambda_N$ be the roots of the equation
$$
\det\left( (I + A) - \lambda (I -B) \right) = 0.
$$
Then $\hat J_N = \lambda_N$. Note that $ \lambda_1 > 0$ by the positive definiteness of  $I + A$ and $I -B$. The relevant maximum is attained at a generalized eigenvector
 $Z$ that can be found from the relation $(I + A)Z = \lambda_N(I - B)Z$ \cite{5}.

By Lemmas \ref{lem2.1} and \ref{lem2.2} the roots of
$$
\det\left( 4x^2(I + A) - (I - B) \right) = 0
$$
are $\left\{  \pm \mu_j \right\}_{j=1}^{\lfloor\bfrac{N+1}{2}\rfloor}$, $\left\{  \pm \nu_j \right\}_{j=1}^{N - \lfloor\bfrac{N+1}{2}\rfloor}$, where
$\mu_j = \cos \frac{j \pi}{N+2}$, $U'_{N+1}(\nu_j) = 0$, and they can be arranged so that 
\begin{alignat*}{2}
&0 < \mu_{\bfrac{N+1}{2}} < \nu_{\bfrac{N-1}{2}} < \cdots < \nu_1 < \mu_1 \quad & &\text{for  $N$ odd},\\
&0 < \nu_{\xfrac{N}{2}} < \mu_{\xfrac{N}{2}} < \cdots < \nu_1 < \mu_1           &  &\text{for $N$ even.}
\end{alignat*}
Consequently, for $N$ odd we have $\lambda_N = \frac{1}{4\mu^2}$, where $\mu = \cos \frac{\frac{N+1}{2}\pi}{N+2} = \sin \frac{\pi}{2(N+2)}$.
For $N$ even, $\lambda_N = \frac{1}{4\nu^2}$, where $U'_{N+1}(\nu) = 0$. Using \eqref{e6}, we obtain $(N + 3) U_N(\nu) - 
(N + 1) U_{N+2}(\nu) = 0$. Setting $\nu = \cos \vartheta_1$, we can write $(N + 3) \sin\,(N+1)\vartheta_1 - (N + 1)\sin\,(N+3)\vartheta_1 = 0$,
and the substitution $\vartheta = \vartheta_1 - \xfrac{\pi}{2}$ ($\nu = \sin \vartheta$) leads to $(N + 3)\cos(N+1)\vartheta + (N +1)\cos(N+3)\vartheta = 0$, completing the proof
\end{proof}

\section{Finding the extremal polynomials of the auxiliary problem}

We have to find the coefficients of the extremal polynomials of \eqref{e7}.

\begin{theorem}
Let $\mu = \sin \frac{\pi}{2(N + 2)}$ be the least positive root of $U_{N+1}(x)$, and $\nu$ the least positive root of $U'_{N+1}(x)$. Let
\begin{align}
&z_k^{(0)} = (-1)^{k-1}\bigg( U_{k-1}(\mu)U_k(\mu) - U_{N-k}(\mu)U_{N-k+1}(\mu) + \frac{N+1-2k}{N+1}U_N(\mu)U_{N+1}(\mu) \bigg), \label{e10}\\
&z_k^{(1)} = (-1)^{k-1}\bigg( U_{k-1}(\nu)U_k(\nu) - U_{N-k}(\nu)U_{N-k+1}(\nu) + \frac{N+1-2k}{N+1}U_N(\nu)U_{N+1}(\nu) \bigg) \label{e11}                                                                                 
\end{align}
for $k = 1, \ldots, N$,  be the coordinates of the principal vectors of the matrix pencils $I + A - \frac{1}{4\mu^2}(I -B)$, $I + A - \frac{1}{4\nu^2}(I -B)$ respectively.
Then the coefficients of the extremal polynomial of problem \eqref{e7} are given by
\begin{equation}\label{e12}
\begin{aligned}[c]
&\alpha_k = \frac{ \sum_{j=1}^{N-k+1}z_j^{(s)} z_{j+k-1}^{(s)} - \sum_{j=1}^{N-k-1} z_j^{(s)} z_{j+k+1}^{(s)} }
                         { \sum_{j=1}^{N} z_j^{(s)} z_j^{(s)} - \sum_{j=1}^{N-2} z_j^{(s)} z_{j+2}^{(s)} },\quad k = 1, \ldots, N-2, \\ 
&\alpha_{N-1} = \frac {z_1^{(s)} z_{N-1}^{(s)} + z_2^{(s)} z_N^{(s)}} { \sum_{j=1}^N z_j^{(s)} z_j^{(s)} - \sum_{j=1}^{N-2} z_j^{(s)} z_{j+2}^{(s)} }, \\ 
&\alpha_{N} = \frac {z_1^{(s)} z_N^{(s)}} { \sum_{j=1}^N z_j^{(s)} z_j^{(s)} - \sum_{j=1}^{N-2} z_j^{(s)} z_{j+2}^{(s)} }.
\end{aligned}
\end{equation}
where $s=0$ for odd $N$ and $s=1$ if $N$ is even.
\end{theorem}

\begin{proof}
Let $F_N^{(0)}(z) = \sum_{j=1}^N \alpha_j z^j$ be the extremal polynomial and  $\Im\left( F_N^{(0)} (e^{it}) \right) = \beta_0 \sin t \,\left(  1 +
2 \sum_{j=1}^{N-1} \beta_j \cos jt  \right)$. 
The vector $\beta = (\beta_0, \ldots, \beta_{N-1})^T$ is related to the coefficients of $F_N^{(0)}(z)$ by formulas \eqref{e8}, which can be reverted:
\begin{equation}\label{e13}
\alpha_j = \beta_{j-1} - \beta_{j+1},\ j = 1, \ldots, N-2,\quad \alpha_{N-1} = \beta_{N-2},\quad \alpha_N = \beta_{N-1}.
\end{equation}
In particular, $\beta_0 = \frac{1}{1-\beta_2}$. 
Consider a principal vector (not necessarily normalized)
$Z = (z_1, \ldots, z_N)^T$, maximizing the Rayleigh quotient. Then $\beta_k = \sum_{j=1}^{N-k} z_j z_{j+k}$, $k = 1, \ldots, N-1$. Now formulas \eqref{e13} imply \eqref{e12}. 
\end{proof}
 
\begin{corollary}\label{cor4.1.1}
For every $N = 1, 2, \ldots$ the extremal polynomial of problem \eqref{e7} is unique.
\end{corollary}

\begin{proof} For a simple generalized eigenvalue, the corresponding subspace  of principal vectors is one-dimensional.
\end{proof}

For $N$ odd, formulas \eqref{e10} and \eqref{e12} can be simplified.

\begin{theorem}
Let $N$ be odd. Then the coefficients of the extremal polynomial are given by
\begin{equation}\label{e14}
\alpha_k = \frac {(-1)^{k-1}} { U'_N \left( \sin \frac {\pi} {2(N+2)}  \right) }  U'_{N-k+1} \biggl( \sin \frac {\pi} {2(N+2)} \bigg) U_{k-1} \left( \sin \frac {\pi} {2(N+2)} \right), k = 1, \ldots, N.
\end{equation}
\end{theorem}

\begin{proof}
Consider the principal vector $Z^{(0)} = \gamma \left\{ (-1)^{k-1} U_{k-1}(\mu) U_k(\mu) \right\}_{k=1}^N$, where $\gamma$
is determined from the normalization $\sum_{j=1}^N z_j^{(0)} z_j^{(0)} - \sum_{j=1}^{N-2} z_j^{(0)} z_j^{(0)} = 1$. To find $\gamma$, we write the normalization condition as
\begin{multline*}
\gamma^2\bigl(U_0U_0U_1U_1 + \cdots + U_{N-1} U_{N-1} U_N U_N \\- (U_0 U_1 U_2 U_3 + \cdots
               + U_{N-3} U_{N-2} U_{N-1} U_N)\big) = 1,
\end{multline*}
or
$$
\gamma^{-2} = \sum_{j=1}^N \left( U_{j-1}(\mu) U_j(\mu)  \right)^2 - \sum_{j=1}^{N-2} U_{j-1}(\mu)U_j(\mu)U_{j+1}(\mu)U_{j+2}(\mu).
$$
Using $U_{j-1}(x)U_{j+1}(x) = \left( U_j(x) \right)^2 - 1$ and $U_{N-j}(\mu) = (-1)^{ \bfrac{N-1}{2} }  U_j(\mu)$ we get
\begin{align*}
\gamma^{-2} &= \left( U_0(\mu) U_1(\mu) \right)^2 + \sum_{j=1}^{N-2} \left( U_j(\mu) U_{j+1}(\mu) \right)^2 + \left( U_{N-1}(\mu) U_N(\mu) \right)^2  \\
\
&\quad -
\sum_{j=1}^{N-2}\bigl( ( U_j(\mu) U_{j+1}(\mu) )^2 - ( U_j(\mu) )^2 - ( U_{j+1}(\mu) )^2 +1 \big) \\
&=
2 \sum_{j=1}^{N-1}\left( U_j(\mu) \right)^2 - (N - 2)  \\
&= \frac {1} {\cos^2 \frac {\pi} {2(N+2)} } \biggl( 2 \sum_{j=1}^{N-1} \sin^2 \frac {\pi(N + 1)(j + 1)} {N + 2} - (N - 2)\cos^2 \frac {\pi} {2(N + 2)}  \bigg)  \\
&= \frac {1} {\cos^2 \frac {\pi} {2(N+2)} } \biggl( N - 2 \cos \frac {\pi} {N + 2} - (N - 2) \cos^2 \frac {\pi} {2(N + 2)} \bigg)\\
& = (N + 2) \tan^2 \frac {\pi} {2(N + 2)}. 
\end{align*}
Hence \[
  \gamma = \frac{1}{\sqrt{N + 2}} \cot \frac {\pi} {2(N +2)}\]
and
\begin{align*}
\alpha_k = \frac {(-1)^{k-1}} {N + 2} \cot^2
\frac {\pi} {2(N + 2)}\cdot
\Big( &\sum_{j=1}^{N-k+1} U_{j-1}(\mu) U_j(\mu) U_{j+k-2}(\mu) U_{j+k-1}(\mu)  \\
       -& \sum_{j=1}^{N-k-1} U_{j-1}(\mu) U_j(\mu) U_{j+k}(\mu) U_{j+k+1}(\mu) \Big), \\
        & \quad k = 1, \ldots, N.
\end{align*}
It turns out that the products in the second sum can be transformed in such a way that most of them cancel with the corresponding products in the first sum.
We will use the easily verifiable identity
$$
U_m(x) U_n(x) = U_{m+1}(x) U_{n-1}(x) - U_{n-m-2}(x).
$$

Then
\begin{multline*}
    \sum_{j=1}^{N-k-1} U_{j-1}(\mu) U_{j}(\mu) U_{j+k}(\mu) U_{j+k+1}(\mu)  \\
\begin{aligned}
  &=\sum_{j=1}^{N-k-1} U_j(\mu) U_{j+1}(\mu) U_{j+k-1}(\mu) U_{j+k}(\mu)  \\
  &\quad- \sum_{j=1}^{N-k-1} U_j(\mu) U_{j+k-1}(\mu) U_{k-1}(\mu) - \sum_{j=1}^{N-k-1} U_{j+1}(\mu) U_{j+k}(\mu) U_{k-1}(\mu)  \\ 
 &\quad+ (N - k - 1) \left( U_{k-1}(\mu) \right)^2.
\end{aligned}
\end{multline*}

Consequently,
\begin{align*}
\alpha_k &= (-1)^{k-1} \gamma^2 \Big( U_1(\mu) U_{k-1}(\mu) U_k(\mu) + U_{N-k}(\mu) U_{N-k+1}(\mu) U_{N-1}(\mu) U_N(\mu)  \\
&\quad+ \sum_{j=1}^{N-k-1} U_{j+1}(\mu) U_{j+k}(\mu) U_{k-1}(\mu) \\ 
&\quad+ \sum_{j=1}^{N-k-1} U_j(\mu) U_{j+k-1}(\mu) U_{k-1}(\mu) - (N - k - 1) (U_{k-1}(\mu))^2 \Big),
\end{align*}

and with the use of $U_{N-j}(\mu) = (-1)^{\bfrac{N-1}{2}} U_j(\mu)$,
$$
\alpha_k = (-1)^{k-1} \gamma^2 U_{k-1}(\mu) \left( 2 \sum_{j=1}^{N-k} U_j(\mu) U_{j+k-1}(\mu) - (N - k -1) U_{k-1}(\mu) \right).
$$
Applying Corollary \ref{cor2.5.1}, the formulas $U_{k-3}(\mu) = 2 \mu U_{k-2}(\mu) - U_{k-1}(\mu)$, $ U_k(\mu) = 2 \mu U_{k-1}(\mu) - U_{k-2}(\mu)$, 
and \eqref{e6} we deduce that
\begin{align*}
\alpha_k &= (-1)^{k-1} U_{k-1}(\mu) \frac{\gamma^2 \mu}{2(1 - \mu^2)} \left( (N - k + 3) U_k(\mu) - (N - k +1) U_{k-2}(\mu) \right) \\
&= (-1)^{k-1} U_{k-1}(\mu) \frac{1}{U'_N(\mu)} U'_{N - k +1}(\mu).\qedhere
\end{align*}
\end{proof}

Note that the relation  $U_{k-1}(\mu) U_k(\mu) = (-1)^{k-1}\left| U_{k-1}(\mu) U_k(\mu) \right|$ implies
\begin{align*}
\alpha_k &= \frac{2 \gamma^2 \mu}
{1 - \mu^2} \bigl( (N- k +3) \left| U_{k-1}(\mu) U_k(\mu) \right| + (N -k +1) \left| U_{k-2}(\mu) U_{k-1}(\mu) \right| \big)\\
& > 0, \quad k = 1, \ldots, N.
\end{align*}

We will show that for  $N$ even the coefficients of the extremal polynomial of problem \eqref{e7} are positive.

\begin{theorem}
Let $N$ be even. Then the coefficients of the extremal polynomial of problem \eqref{e7} are positive.
\end{theorem}
\begin{proof}
In formulas \eqref{e12}, the denominator is positive; call it $\beta^2$.
 Lemma \ref{lem2.7}(a) yields
  $\alpha_N = \beta^{-2} z_1^{(1)} z_1^{(1)}$, $
\alpha_{N-1} = 2 \beta^{-2} z_1^{(1)} z_2^{(1)}$, and 

\begin{multline*}
  \alpha_k = \beta^{-2} \Bigl(
    \sum_{j=1}^{N-k-1} z_j^{(1)} ( z_{N-k-j+2}^{(1)} - z_{N-k-j}^{(1)} )\\ +
    z_{N-k}^{(1)} z_2^{(1)} + z_{N-k+1}^{(1)} z_1^{(1)}\Bigr) 
    ,\quad k = \xfrac{N}{2}, \ldots, N-2. 
\end{multline*}

Since $z_{N/{2+1}}^{(1)} = z_{N/2}^{(1)}$, we can assume that all indices in the above formula are no greater than $\xfrac{N}{2}$. By Lemma \ref{lem2.7}(b) we have $\alpha_k > 0$, $k = \xfrac{N}{2}, \ldots, N$.

Let now $k = \xfrac{N}{2} -1$. Then 
\[
\beta^2 \alpha_{N/2-1} = \sum_{j=1}^{n/2+2} z_j^{(1)} z_{j+N/2-2}^{(1)} - \sum_{j=1}^{N/2} z_j^{(1)} z_{j+N/2}^{(1)}.
\]
Grouping terms as before will not work, because already  $z_1^{(1)} z_{N/2-1}^{(1)} - z_1^{(1)} z_{N/2+1}^{(1)}\allowbreak = z_1^{(1)} \left( z_{N/2-1}^{(1)} -
z_{N/2}^{(1)} \right) < 0$. 
One can, however, combine the  first two summands in each sum:
\begin{align*}
     &z_1^{(1)} z_{N/2-1}^{(1)} + z_2^{(1)} z_{N/2}^{(1)} - z_1^{(1)} z_{N/2+1}^{(1)} -z_2^{(1)} z_{N/2-1}^{(1)} \\
     & = z_1^{(1)} z_{N/2-1}^{(1)} + z_2^{(1)} z_{N/2}^{(1)} - z_1^{(1)} z_{N/2}^{(1)} -  z_2^{(1)} z_{N/2-1}^{(1)} \\
     & = \left( z_{N/2}^{(1)} - z_{N/2-1}^{(1)} \right) \left( z_2^{(1)} - z_1^{(1)} \right) > 0.
\end{align*}
Then
\begin{align*}
\alpha_{N/2-1} &= \beta^{-2} \Bigl( ( z_{N/2}^{(1)} - z_{N/2-1}^{(1)} ) ( z_2^{(1)} - z_1^{(1)} ) + \sum_{j=3}^{N/2} z_j^{(1)} ( z_{N/2-j+3}^{(1)} - z_{N/2-j+1} ^{(1)})  \\
&\quad + z_{N/2}^{(1)} z_2^{(1)} + z_{N/2-1}^{(1)} z_1^{(1)} \Big).
\end{align*}

Hence $\alpha_{N/2-1} > 0$ by Lemma \ref{lem2.7}(b).

Let $k = 2, \ldots, \xfrac{N}{2} - 2$. Then
\begin{align*}
\alpha_k = \beta^{-2} \Bigl( \sum_{j=1}^{N-k-1} z_j^{(1)} z_{j+k-1}^{(1)} - \sum_{j=1}^{N-k-1} z_j^{(1)} z_{j+k+1}^{(1)}
+z_{k+1}^{(1)} z_2^{(1)} + z_k^{(1)} z_1^{(1)}  \Big).
\end{align*}
Grouping the first two terms in each sum we get
\begin{multline}\label{eq15}
  \alpha_k = \beta^{-2} \Big( \sum_{j=1}^{N/2-k-1} ( z_j^{(1)} z_{j+k-1}^{(1)} -z_j^{(1)} z_{j+k+1}^{(1)} )\\
\begin{aligned}[b]
  &+
                                            \sum_{j=N/2-k+2}^{N-2k} ( z_j^{(1)} z_{j+k-1}^{(1)} - z_j^{(1)} z_{j+k+1}^{(1)} )  \\
                                           &+ z_{N/2-k}^{(1)} z_{N/2-1}^{(1)} + z_{N/2-k+1}^{(1)} z_{N/2}^{(1)} - z_{N/2-k}^{(1)} z_{N/2+1}^{(1)} - z_{N/2-k+1}^{(1)} z_{N/2+2}^{(1)}   \\
                                           &+ \sum_{j=N-2k+1}^{N-k-1} ( z_j^{(1)} z_{j+k-1}^{(1)} - z_j^{(1)} z_{j+k+1}^{(1)} ) + z_{k+1}^{(1)} z_2^{(1)} + z_k^{(1)} z_1^{(1)}              
                                  \Big).
\end{aligned}
\end{multline}

The second sum in \eqref{eq15} can be transformed to
$$
\sum_{j=1}^{N/2-k-1}\left( z_{N-2k+1-j}^{(1)} z_{N-k-j}^{(1)} - z_{N-2k+1-j}^{(1)} z_{N-k+2-j}^{(1)} \right).
$$
Since $N - k - j > \xfrac{N}{2}$, we have $z_{N-k-j}^{(1)} = z_{k+1+j}^{(1)}$, $z_{N-k+2-j}^{(1)} = z_{k-1+j}^{(1)}$. 
The first two sums in \eqref{eq15} can be written as
\begin{align*}
      &\sum_{j=1}^{N/2-k-1} \left( z_j^{(1)} z_{j+k-1}^{(1)} - z_j^{(1)} z_{j+k+1}^{(1)} + z_{N-2k+1-j}^{(1)} z_{j+k+1}^{(1)} - z_{N-k+1-j}^{(1)} z_{j+k-1}^{(1)}\right) \\
   = &\sum_{j=1}^{N/2-k-1} \left( z_{j+k+1}^{(1)} - z_{j+k-1}^{(1)} \right) \left( z_{N-2k+1-j}^{(1)} - z_j^{(1)} \right) \\
   = &\sum_{j=1}^{N/2-k-1} \left( z_{j+k+1}^{(1)} - z_{j+k-1}^{(1)} \right) \left( z_{\varphi_j}^{(1)} - z_j^{(1)}\right),
\end{align*}
where
$$
\varphi_j =
\begin{cases} 
      N - 2k + 1 - j,  & 2k + j > \xfrac{N}{2}, \\
      2k + j,             & 2k + j \leq \xfrac{N}{2}. 
   \end{cases}
$$
Since $j < \varphi_j \leq \xfrac{N}{2}$, we obtain
\begin{align}
\alpha_k = \beta^{-2} \Big( &\sum_{j=1}^{n/2-k-1} ( z_{j+k+1}^{(1)} -z_{j+k-1}^{(1)} ) ( z_{\varphi_j}^{(1)} - z_j^{(1)} ) \notag\\ 
&+( z_{N/2-k+1}^{(1)} - z_{N/2-k}^{(1)} ) ( z_{n/2}^{(1)} - z_{N/2-1}^{(1)} ) \notag \\
                                          &+\sum_{j=1}^{k-1} z_{N-2k+j}^{(1)} ( z_{k+2-j}^{(1)} - z_{k-j}^{(1)} ) + z_{k+1}^{(1)} z_{2}^{(1)} + z_k^{(1)} z_1^{(1)}
                                  \Big) \label{e16}
\end{align}

for $k = 2, \ldots, \xfrac{N}{2} - 2$. 
In the last sum, change $z_{N-2k+j}^{(1)}$ to $z_{\Psi_j}^{(1)}$, where
$$ \Psi_j = \begin{cases}
                    N - 2k + j,  & 2k + 1 - j  > \xfrac{N}{2}, \\
                    2k + 1 - j,   & 2k + 1 - j \leq \xfrac{N}{2}.
\end{cases}
$$
Then all indices in \eqref{e16} are $\le \xfrac{N}{2}$. 
By Lemma \ref{lem2.7}(b), all the differences in brackets are positive, so
 $\alpha_k > 0$, $k = 2, \ldots, \xfrac{N}{2}-2$. 
\end{proof}

\section{Solution of the main problem}

The coefficients of the extremal polynomials of problem \eqref{e7} are positive, so these polynomials are also extremal for~\eqref{e1}.
Moreover, by 
Corollary \ref{cor4.1.1} the extremal polynomial of \eqref{e7} is unique.
Hence \eqref{e1} has exactly two extremal polynomials related by 
 $F_N^{(2)} = -F_N^{(1)}(-z)$, where $F_N^{(1)}(z)$
is the extremal polynomial of \eqref{e7} whose coefficients can be determined from \eqref{e10}--\eqref{e12}.

\section{Examples}
We will construct some extremal polynomials and find the images of the unit circle under them.
Consider $N = 6$ and $N = 7$.
For $N$ even we find the coefficients from \eqref{e11}, \eqref{e12}, while for $N$ odd, from \eqref{e10}, \eqref{e12} or \eqref{e14}:  
$$
F_6(z) = z +
\sum_{j=2}^{6} \alpha_j z^j, 
$$ 
where $\alpha_2 = 1.36252\ldots,$ $\alpha_3 = 1.55595\ldots,$ $\alpha_4 = 1.22943\ldots,$ $\alpha_5 = 0.84332\ldots,$
$\alpha_6 = 0.37361\ldots$; and
$$
F_7(z) = z +\sum_{j=2}^{7} b_j z^j,
$$
where $b_2 = 1.44834\ldots,$ $b_3 = 1.77398\ldots,$ $b_4 = 1.55232\ldots,$ $b_5 = 1.32706\ldots,$
$b_6 = 0.75810\ldots,$ $b_7 = 0.43104\ldots$.
We compute $\nu = 0.19818\ldots,$ $\mu = 0.17364\ldots,$ $J_6 = F_6(1) = \frac{1}{4\nu^2} = 6.36485\ldots,$ $J_7 = F_7(1) = \frac{1}{4\mu^2} =
8.29085\ldots.$ The images of the 
upper unit semicircle under the maps $F_6(z)$ and $F_7(z)$ are shown in Figures \ref{fig1} and \ref{fig2}.

\begin{figure}[H]
\centering
\includegraphics[scale=0.25]{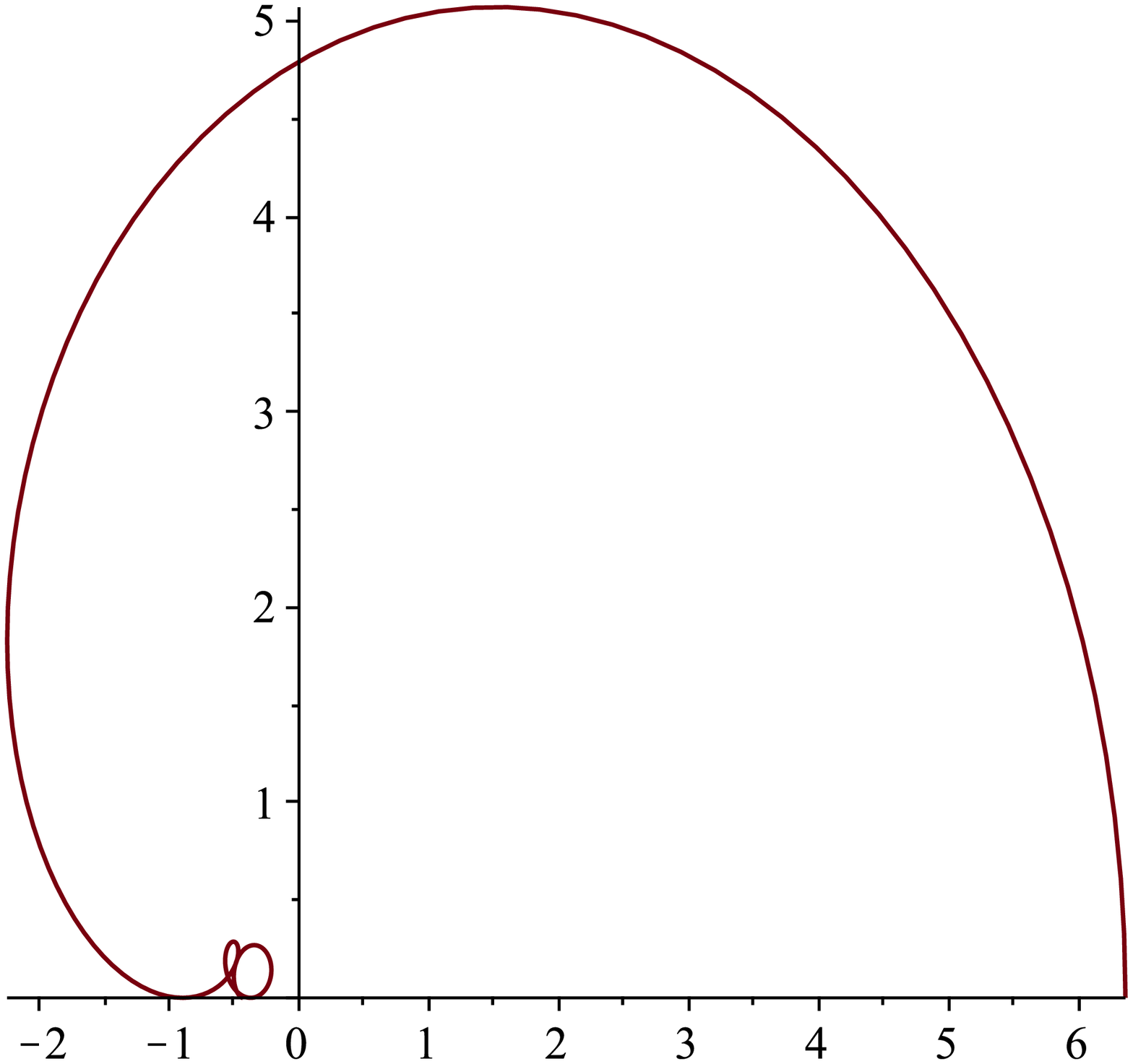}
\caption{The image of the upper unit semicircle under $F_6(z)$ \label{fig1}}
\end{figure}

\begin{figure}[H]
\centering
\includegraphics[scale=0.25]{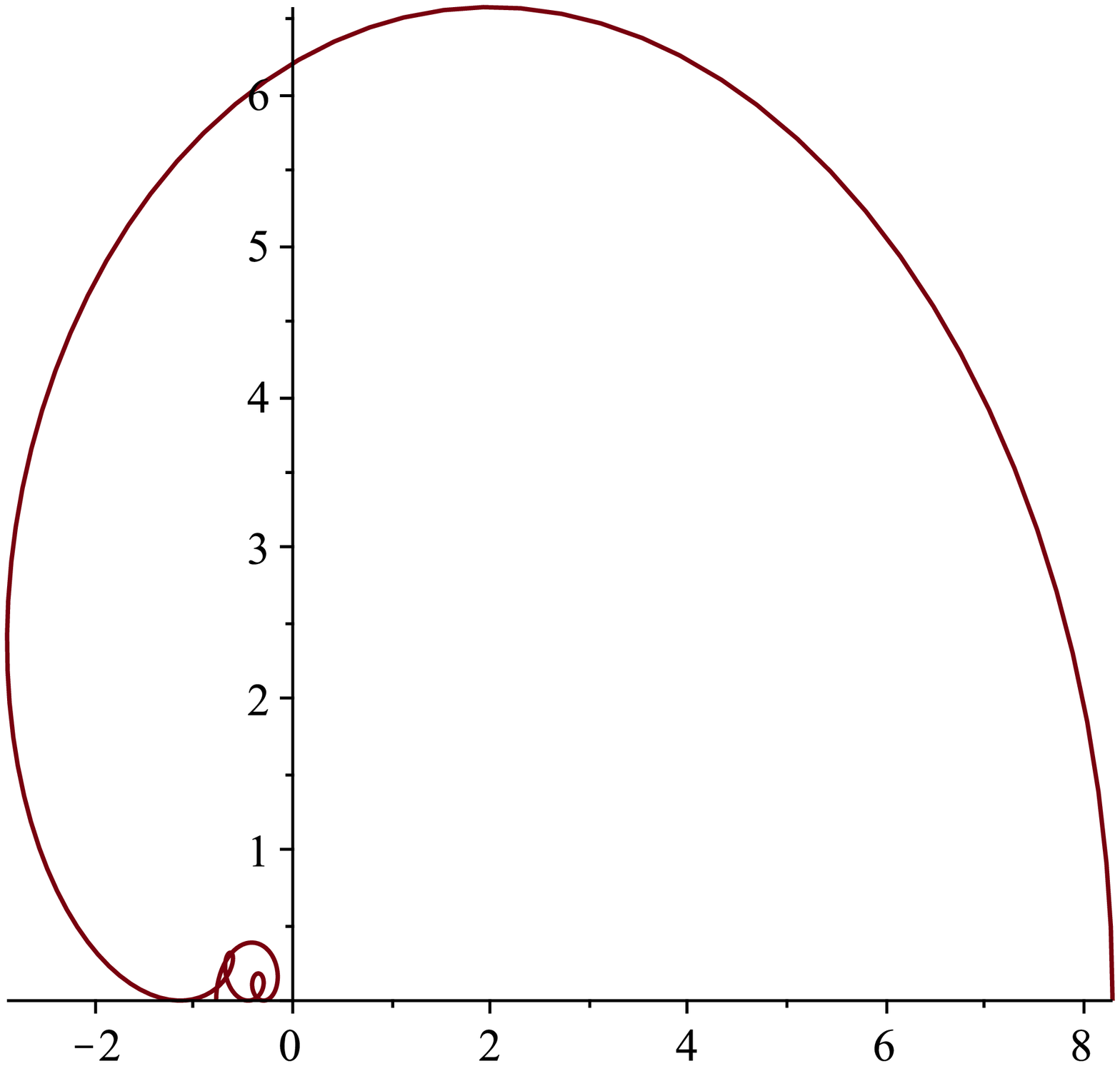}
\caption{The image of the upper unit semicircle under $F_7(z)$ \label{fig2}}
\end{figure}

\section{Conclusion}

Extremal problems for typically real polynomials
 $z + \sum_{j=2}^{N} \alpha_j z^j$  are equivalent to the same problems for sine polynomials  $\sin t + \sum_{j=2}^{N} \alpha_j \sin jt$, nonnegative on~$ [0, \pi]$. In \cite{6} the following extremal problems were considered:
\begin{equation}\label{e17}
\begin{aligned}[c]
& \max \left\{ \alpha_k : \sin t + \sum_{j=2}^N \alpha_j \sin jt \geq 0,\, t \in [0, \pi]  \right\}, \\ 
& \min \left\{ \alpha_k : \sin t + \sum_{j=2}^N \alpha_j \sin jt \geq 0,\, t \in [0, \pi] \right\},\quad k = 2, \ldots, N,
\end{aligned}
\end{equation}
and also
\begin{equation}\label{e18}
\max \left\{\! \sin \vartheta + \sum_{j=2}^N \alpha_j \sin j\vartheta: \sin t + \sum_{j=2}^N \alpha_j \sin jt \geq 0,\, t \in [0, \pi] \right\},\ \vartheta \in [0, \pi].\ 
\end{equation}
Moreover, the problem
\begin{equation}\label{e19}
\max \left\{ {-}1 + \sum_{j=2}^N  (-1)^j \alpha_j: \sin t + \sum_{j=2}^N \alpha_j \sin jt \geq 0,\, t \in [0, \pi] \right\}
\end{equation}
was considered in \cite{7}, and the problem
 \begin{equation}\label{e20}
\max \left\{ 1 + \sum_{j=2}^N |\alpha_j|: \sin t + \sum_{j=2}^N \alpha_j \sin jt \geq 0,\, t \in [0, \pi] \right\}
\end{equation}
in \cite{1}.
Problem \eqref{e20} coincides with problem \eqref{e1}.

In \cite{6}, extremal values were found for problem \eqref{e17} for   $k = 2, 3, N-1, N$ and for problem 
\eqref{e18}.
 In \cite{8}, extremal values for problem \eqref{e17} were found for $k = N - 2$. Finding extremal polynomials turned out to be much harder.
For \eqref{e20} with $k = N-1, N$, extremal polynomials were found in  \cite{8}, as also was the extremal polynomial for $k = N - 2$ with $N$ odd. 
For \eqref{e17} with  $k = 2, 3$ and for \eqref{e18} extremal polynomials are unknown.
 In \cite{7} problem \eqref{e19} was completely solved.  In \cite{3} it was solved in another way, which enabled the proof of uniqueness of the extremal polynomial; the problem was also generalized to arbitrary polynomials, not necessarily typically real.

In \cite{9} a relation was found between \eqref{e19} and the Koebe problem for polynomials.
Note that problem \eqref{e1} in the class of univalent polynomials was solved in
\cite{2}: the extremal value is $\frac14 \left(1 + \frac{1}{N}\right)  \csc^2 \frac {\pi} {2(N + 1)}$, attained at a Suffridge polynomial \cite{10}.

A subject for future research is to find a formula for the coefficients of extremal polynomials for $N$ even in the form analogous to \eqref{e14}.

Finally, let us note that some extremal typically real polynomials can be used in the problems of stability in discrete dynamical systems \cite{12}. 

\section*{Acknowledgements}
The authors are grateful to Elena Berdysheva (Justus Liebig University Giessen) and Paul Hagelstein (Baylor University) for interesting and useful discussions, 
and to Jerzy Trzeciak (IMPAN) for his help in preparing the manuscript.


\begin{thebibliography}{10}
                                      
\bibitem{7} 
Brandt M., Variationsmethoden für in der Einheitskreisscheibe schlichte Polynome, 
Dissertation, Humboldt-Univ. Berlin,  1987.

\bibitem{2} 
Brandt M., 
Representation formulas for the class of typically real polynomials, 
Math. Nachr. 144 (1989), 29--37.
                                              
\bibitem{8} 
Dimitrov D. K. and  Merlo C. A.,  
Nonnegative trigonometric polynomials,  
Constr. Approx. 18 (2002),  117--143.

\bibitem{9} 
Dmitrishin D., Dyakonov K. and Stokolos A., 
Univalent polynomials and Koebe's one-quarter theorem,
Anal. Math. Phys. 9 (2019),  991–-1004.

\bibitem{12} 
Dmitrishin, D., Hagelstein, P., Khamitova, A., Korenovskyi, A., and Stokolos, A., Fejer polynomials and control of nonlinear discrete systems, Constr. Approx. 51 (2020). 383--412.

\bibitem{3} 
Dmitrishin D., Smorodin A. and Stokolos A., 
On a family of extremal polynomials, 
C. R. Math. Acad. Sci. Paris 357 (2019), 591--596.
                                            
\bibitem{4} 
Fej\'er L., 
Ueber trigonometrische Polynome,
J. Reine Angew. Math. 146 (1915), 53--82.

\bibitem{11}
Gantmacher F. R.,
The Theory of Matrices,
Vol. I, AMS Chelsea Publ., Amer. Math. Soc., 2000.
                                                  
\bibitem{5} 
Martin R. S.  and Wilkinson J. H., 
Reduction of  the symmetric eigenproblem $Ax=\lambda Bx$ and related problems to standard form, 
Numer. Math. 11 (1968), 99--110.
                                                 
\bibitem{1} 
Michel C.,  
Untersuchungen zum Koeffizientenproblem bei schlichten Polynomen, 
Dissertation, Humboldt-Univ. Berlin,  1971.

\bibitem{6} 
Rogosinski W. W.  and Szeg\H{o} G., 
Extremum problems for non-negative sine polynomials, 
Acta Sci. Math. (Szeged) 12 (1950),  112–124.

\bibitem{10} 
Suffridge T., 
On univalent polynomials, 
J. London Math. Soc. 44 (1969),  496--504.
                                                  
\end{thebibliography}
\end{document}